\documentclass[reqno,tbtags,intlimits,a4paper,oneside,11pt]{amsart}
\usepackage{amsfonts}
\usepackage{}
\usepackage{amsfonts,amsmath,amssymb}
\newtheorem{Theorem}{Theorem}[section]
\newtheorem{prop}[Theorem]{Proposition}
\newtheorem{Lemma}[Theorem]{Lemma}

\newtheorem{Remark}[Theorem]{Remark}

\def\beq#1#2\eeq{%
        \begin{equation}%
        \label{#1}%
            #2%
        \end{equation}%
    }

\usepackage{color}
\usepackage{graphicx}
\usepackage{epstopdf}

\title[Geodesic scattering on hyperboloids and Kn\"orrer's map ]{Geodesic scattering on hyperboloids and Kn\"orrer's map}

\author{A.P. Veselov}
\address{Department of Mathematical Sciences,
Loughborough University, Loughborough LE11 3TU, UK, Moscow State University and Steklov Mathematical Institute, Moscow, Russia}
\email{A.P.Veselov@lboro.ac.uk}

\author{L. Wu}\address{School of Mathematical Science\\
Huaqiao University, Quanzhou,
Fujian, P. R. China 362021}
\email{wulihua@hqu.edu.cn}

\begin{document}

\maketitle

\begin{abstract}
We use the results of Moser and Kn\"orrer on relations between geodesics on quadrics and solutions of the classical Neumann system to describe explicitly the geodesic scattering on hyperboloids.

We explain the relation of Kn\"orrer's reparametrisation with projectively equivalent metrics on quadrics introduced by Tabachnikov and independently by Matveev and Topalov, giving a new proof of their result.
We show that the projectively equivalent metric is regular on the projective closure of hyperboloids and extend Kn\"orrer's map to this closure.
\end{abstract}


\section{Introduction}

Geodesic flows on ellipsoids have been a subject of substantial interest since the fundamental work of Jacobi \cite{Jac} (see in particular, discussion in Arnold's book \cite{A}). Moser  reinvigorated this area in 1970s by showing the deep relation to both classical geometry, spectral and soliton theory \cite{M1,M2,M3}.

The geodesic flow on the hyperboloids, to the best of our knowledge, was not studied in such details, partly because the generalisation from the ellipsoid case seems to be obvious.

Nevertheless, there are some interesting questions which we could not find the answer in the literature. Consider, for example, one-sheeted hyperboloid in $\mathbb R^3$ given by
\begin{equation}\label{eq 0}
Q(x):=\frac{x_0^2}{a_0}+\frac{x_1^2}{a_1}+\frac{x_2^2}{a_2}=1
\end{equation}
with $a_0<0<a_1<a_2$.
There are three obvious geodesics given by the intersections with the coordinate planes, one of them is the elliptical neck of hyperboloid, given by $x_0=0$.
In fact, this is the only closed geodesic in this case, all other geodesics are unbounded. Since at the infinity the hyperboloid is close to its asymptotic cone $Q(x)=0$,
which is flat, we know that such geodesic should have asymptotic velocity $y$ at infinity, satisfying the relations
$$
Q(y)=0, \, |y|=1.
$$

There are two natural questions we would like to address.

{\it Question 1. Consider a geodesic coming from the infinity with, say, positive $x_0$. Will it be reflected back, or will it pass through to the infinity with negative $x_0$?
How many times will it rotate around the hyperboloid?}

As we will see there is also an intermediate case, when geodesics stuck spiralling around the neck (see Fig. 1).

\begin{figure}[h]
  \includegraphics[width=30mm]{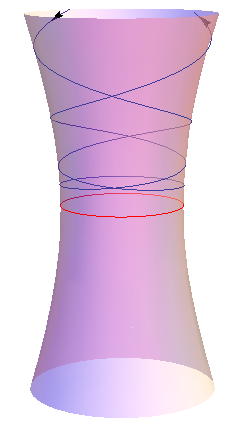}\quad
  \includegraphics[width=35mm]{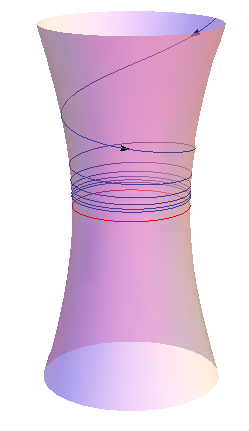}\quad
  \includegraphics[width=37mm]{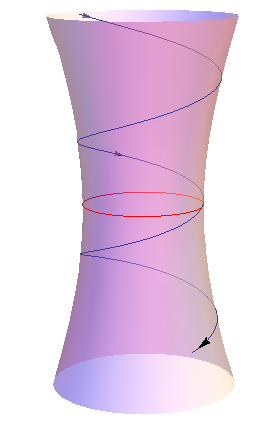}
  \caption{Geodesic scattering on one-sheeted hyperboloid}
\end{figure}

{\it Question 2. Is there an explicit relation between the values of asymptotic velocities at $\pm \infty$ in terms of the corresponding integrals?}

Both questions can be easily answered for the symmetric case $a_1=a_2$, using the Clairaut integral for the geodesics on the surface of revolution (see Section 7 below).

We will consider the general case of the geodesic scattering on hyperboloids in $\mathbb R^{n+1}$ using Moser's analysis  \cite{M1,M2,M3} of Jacobi problem and Kn\"orrer's results \cite{Kn} on relation of this problem with the classical Neumann problem of motion on sphere in harmonic field \cite{Neu}.

In particular, with regard to Question 1 we have the following result.

\begin{Theorem}
A geodesic on one-sheeted hyperboloid $$\frac{x_0^2}{a_0} + \dots +\frac{x_n^2}{a_n} =1$$ in $\mathbb R^{n+1}$ with $a_0<0<a_1<\dots <a_n$ is crossing the neck ellipsoid with $x_0=0$  if and only if the integral $$F_0=y_0^2+\sum_{j=1}^n\frac{(x_ky_j-x_jy_k)^2}{a_0-a_j}>0.$$
When $F_0<0$ the geodesic is reflected back to infinity, while for $F_0=0$ the geodesic exponentially approaches a geodesic on the neck ellipsoid.
\end{Theorem}

We describe the scattering map explicitly in terms of the real abelian integrals (Theorem \ref{scattering}). Note that in the complex case the situation is much more tricky, see e.g. the discussion of this issue in relation with the billiard on an ellipsoid in \cite{V0}. We briefly discuss the quantum situation in the simplest case in section 8.

We start with a review of the celebrated results of Moser and Kn\"orrer presenting all the necessary calculations. The geodesic scattering is described explicitly in sections 5 and 6 with particular analysis of two-dimensional situation. In sections 7 and 8 we briefly discuss the classical and quantum problem in the symmetric 2D case.

In section 9 we explain the relation of Kn\"orrer's map with projectively equivalent metrics on quadrics discovered by Tabachnikov \cite{T,T2} and independently by Matveev and Topalov \cite{MT,TM}, giving a new proof of their result. We show that in the hyperboloid case the projectively equivalent metric (in contrast to the usual induced one) is regular on the projective closure of hyperboloid $\mathcal H \subset \mathbb RP^{n+1}$ and extend Kn\"orrer's map to it, completing the relation between geodesics and solutions of Neumann system in the hyperboloid case.


\section{Jacobi geodesic problem after Moser}\setcounter{equation}{0}

We follow here closely Moser's work \cite{M1,M2,M3} replacing ellipsoid by
a one-sheeted $n$-dimensional hyperboloid in the Euclidian space $\mathbb{R}^{n+1}$ given by
\begin{equation}\label{eq 2.1}
\begin{array}{rcl}
Q(x):=(A^{-1}x,x)=1
\end{array}
\end{equation}
where  $A=diag(a_0,a_1,\cdots,a_n)$ with
$a_0<0<a_1<\dots<a_n$.

The geodesics on it in the natural parameter $s$ (the length of arc) are
\begin{equation}\label{eq 2.2}
\begin{array}{rcl}
x^{''}=-\lambda Bx,
\end{array}
\end{equation}
where $B=A^{-1},\ '=\frac{d}{ds},$ $\lambda$ is Lagrange multiplier, which is determined by the constraint $(Bx,x)=1$:
$$
(Bx,x')=0, \, \, (Bx',x')+(Bx,x'')=0, \,\, (Bx',x')-(Bx,\lambda Bx)=0,
$$
which implies the formula for Lagrange multiplier
\beq{lambda}
\lambda=\frac{(Bx',x')}{(Bx,Bx)}.
\eeq

The following result goes back to Joachimsthal (see Section 3 in \cite{M2}).

\begin{prop} The function
\begin{equation}\label{eq 2.5}
\begin{array}{rcl}
J(x,x')=(Bx,Bx)(Bx',x'),
\end{array}
\end{equation}
is an integral of the geodesic problem \eqref{eq 2.2}.
\end{prop}

\begin{proof}
We have
\begin{equation}\label{eq 2.6}
\begin{array}{rcl}
J'&=&2(Bx',Bx)(Bx',x')+2(Bx,Bx)(Bx',x'')\\
  &=& 2(Bx',Bx)(Bx',x')-2\lambda(Bx,Bx)(Bx',Bx)\\
  &=&2(Bx',Bx)((Bx',x')-\lambda(Bx,Bx))=0.
\end{array}
\end{equation}
\end{proof}

As an important corollary we see that $$\lambda=\frac{(Bx',x')}{(Bx,Bx)}=\frac{J}{|Bx|^4}$$
does not change sign for a given geodesic $x(s).$

In the ellipsoidal case $J$ is always positive, while in hyperboloid case it may have any sign.
In particular, if Joachimsthal integral $J=0$ we have the generating straight lines on the hyperboloid, lying on the intersection of the hyperboloid with the asymptotic cone with vertex at a given point.

The full set of the integrals of the geodesic problem can be given by the following geometric construction due to Moser \cite{M1,M2}.

Let $x=u+tv$ be a straight line in $\mathbb R^{n+1}$. Then it is easy to check that the condition that this line is tangent to the quadric
$(Bx,x)=1$ is
$$
(1-(Bu,u))(Bv,v)+(Bu,v)^2=0.
$$
Consider now the following function \cite{M2}
\begin{equation}\label{eq phi}
\Phi_z(x,y)=(1+(R_zx,x))(R_zy,y)-(R_z x,y)^2, \quad R_z=(zI-A)^{-1}.
\end{equation}
Geometrically the relation $\Phi_z(u,v)=0$ means that the line $x=u+tv$ is tangent to the confocal quadric
$$
Q_z(x):=((zI-A)^{-1}x,x)+1=0.
$$
This function can be written as
\begin{equation}\label{eq phi2}
\Phi_z(x,y)=\sum_{k=0}^n \frac{F_k(x,y)}{z-a_k},
\end{equation}
where
\begin{equation}\label{eq int}
F_k(x,y)=y_k^2+\sum_{j\neq k}\frac{(x_ky_j-x_jy_k)^2}{a_k-a_j}, \quad k=0, 1,\dots, n.
\end{equation}

As it was explained by Moser \cite{M1,M2}, the following theorem can be derived from the classical geometric results by Chasles \cite{SF}.

\begin{Theorem} (Moser \cite{M1})
The functions $F_k$ given by \eqref{eq int} are the involutive integrals of the geodesic flow \eqref{eq 2.2}, satisfying the relations:
$$
\sum_{k=0}^n a_k^{-1}F_k=0, \quad \sum_{k=0}^n F_k=|y|^2.
$$
\end{Theorem}

Geometrically, this means that a generic line is tangent to $n$ confocal quadrics, such that the corresponding normals are perpendicular to each other, see \cite{M1}.
The corresponding zeros $z=0, c_1,\dots, c_{n-1}$ of $\Phi_z$ are the confocal parameters of these quadrics:
\begin{equation}\label{eq phi3}
\Phi_z(x,y)=\frac{z\prod_{i=1}^{n-1}(z-c_i)}{\prod_{k=0}^{n}(z-a_k)}.
\end{equation}
 Note that since the line is tangent to the initial quadric $z=0$ is always a root:
 \begin{equation}\label{eq phi0}
 \Phi_0(x,y)=-\sum_{k=0}^n a_k^{-1}F_k=0.
 \end{equation}

 According to Moser, the integrals $F_k$ were first written by Uhlenbeck and Devaney. The Joachimsthal integral can be expressed through them as follows:
  \begin{equation}\label{JF}
 J=\frac{d}{dz}\Phi_z(x,y)|_{z=0}=-\sum_{k=0}^n a_k^{-2}F_k.
 \end{equation}
 Indeed, since $1+(R_0x,x)=0=(R_0 x,y)$ the derivative $$\frac{d}{dz}\Phi_z(x,y)|_{z=0}=\frac{d}{dz}R_z(x,x)|_{z=0}(R_0y,y)=|Bx|^2(By,y)=J.$$

Jacobi showed that in the elliptic coordinates $z_1,\dots, z_n$ defined as the non-zero roots of the equation
$$
((zI-A)^{-1}x,x)+1=0
$$
the variables in the Hamilton-Jacobi equation can be separated \cite{Jac}. Moreover, after some reparametrisation, the dynamics becomes linear at the Jacobi variety of the corresponding hyperelliptic curve
$$
w^2=R(z), \quad R(z)=z\prod_{i=1}^{n-1}(z-c_i)\prod_{k=0}^{n}(z-a_k).
$$

We will explain this now in more detail (following again Moser \cite{M1,M2,M3}) for the closely related Neumann problem on the sphere $S^n$ \cite{Neu}, which will play a key role in our considerations.

\section{Neumann integrable system on sphere}

In 1859 Carl Neumann considered the motion on the sphere $|q|^2=1, \, q \in \mathbb{R}^{n+1} $ under the force
with quadratic potential and the Hamiltonian
$$H=\frac{1}{2}|p|^2+\frac{1}{2}(Bq,q), \quad B=diag \, (b_0, \dots, b_n).$$

The corresponding equations of motion are
\beq{neu}
\ddot q=-Bq + \nu q,
\eeq
where $\nu$ is Lagrange multiplier:
$$(\dot q, q)=0, \, (\ddot q, q)+(\dot q, \dot q)=0,\, -(Bq,q)+\nu(q,q)+(\dot q, \dot q)=0,$$
so $\nu=(Bq,q)-(\dot q, \dot q).$

To integrate them Neumann used the Jacobi method of separation of variables in the corresponding Hamilton-Jacobi equation. To do this he introduced the spherical analogue of the Jacobi elliptic coordinates $u_1,\dots, u_n$ on the sphere $S^n$ defined as the roots of the equation
\beq{elki}
\sum_{i=0}^n \frac{q_i^2}{u-b_i}=0, \,\, |q|=1.
\eeq

We have $n$ roots $u=u_j(q)$ satisfying
$$b_0 < u_1 < b_1 < \dots <u_n < b_n.$$

Let $$B(u)= \prod_{i=0}^n (u-b_i), \, \, U(u)= \prod_{j=1}^n (u-u_j),$$ then from (\ref{elki}) we can express
\beq{qii}
q_i^2=\frac{U(b_i)}{B'(b_i)}=\frac{\prod_{j=1}^n(b_i-u_j)}{\prod_{k\neq i}(b_i-b_k)}.
\eeq

Indeed, we have
\beq{F}
\sum_{i=0}^n \frac{q_i^2}{u-b_i}=\frac{\prod_{j=1}^n (u-u_j)}{\prod_{i=0}^n (u-b_i)}=\frac{U(u)}{B(u)},
\eeq
so $q_i^2$ can be found as the residue at $u=b_i.$

Elliptic coordinate system is orthogonal with metric on the sphere having the form
\beq{met}
ds^2=\sum_{j=1}^n g_{jj} du_j^2, \quad g_{jj}=-\frac{U'(u_j)}{4 B(u_j)}
\eeq

Indeed, from (\ref{qii}) we have
$$2q_i^{-1}dq_i=\sum_{j=1}^n \frac{du_j}{u_j-b_i},$$
which gives the following coefficient $g_{ij}$ at $du_idu_j$ in the metric when $i\neq j$:
$$g_{ij}=\frac{1}{4} \sum_{k=0}^n\frac {q_k^2}{(u_i-b_k)(u_j-b_k)}=\frac{-1}{4(u_i-u_j)}(\sum_{k=0}^n\frac {q_k^2}{u_i-b_k}-\sum_{k=0}^n\frac {q_k^2}{u_j-b_k})=0.$$

When $i=j$ we have
$$g_{jj}=\frac{1}{4} \sum_{k=0}^n\frac {q_k^2}{(u_j-b_k)^2}=-\frac{1}{4} \frac{d}{du} \sum_{k=0}^n\frac {q_k^2}{u-b_k}|_{u=u_j}=-\frac{U'(u_j)}{4 B(u_j)}.$$

In the elliptic coordinates the Hamiltonian of the Neumann system has a form
$$H=\frac{1}{2} \sum_{j=1}^n (g_{jj}^{-1} p_j^2 -u_j)=-\frac{1}{2} \sum_{j=1}^n (\frac{4 B(u_j)}{U'(u_j)}p_j^2 +u_j).$$

Indeed, comparing $1/u^2$ terms in (\ref{F}) we have
$$\frac{1}{2}\sum_{i=0}^nb_i q_i^2=\frac{1}{2}\sum_{i=0}^nb_i-  \frac{1}{2}  \sum_{j=1}^n u_j.$$

Thus the Hamilton-Jacobi equation in elliptic coordinates has the form
$$-\sum_{j=1}^n (\frac{4 B(u_j)}{U'(u_j)}S_j^2 +u_j)=2h=2\eta_1, \quad S_j=\frac{\partial S}{\partial u_j}$$

Now we can separate variables using the following identities going back to Jacobi \cite{M1}:
\beq{jaco}
\sum_{j=1}^n \frac{P(u_j)}{U'(u_j)}=\sum_{j=1}^n u_j+2\eta_1, 
\eeq
where $P(z)=z^n+2\eta_1z^{n-1}+ \dots + 2\eta_n$ with arbitrary constants $\eta_i.$

Using this we can rewrite the Hamilton-Jacobi equation as
$$-\sum_{j=1}^n \frac{4 B(u_j)S_j^2+P(u_j)}{U'(u_j)}=0,$$
which can then be separated into $n$ equations
$$(\frac{\partial S}{\partial u_j})^2=-\frac{P(u_j)}{4B(u_j)}.$$

\begin{prop} (Neumann \cite{Neu}, Moser \cite{M1})\label{propos}
The Hamilton-Jacobi equation for the Neumann system has a complete solution of the form
\beq{HJ}
S= \sum_{j=1}^n \int^{u_j} \sqrt{-\frac{P(z)}{4B(z)}}dz, \quad P(z)=z^n+2\eta_1z^{n-1}+ \dots + 2\eta_n
\eeq
with arbitrary constants $\eta_1,\dots,\eta_n.$
The variables
\beq{abeleq1}
\xi_k=-\frac{\partial S}{\partial \eta_k}=\sum_{j=1}^n \int^{u_j} \frac{z^{n-k}}{\sqrt{R(z)}}dz, \,\,\,  R(z)=-4P(z)B(z)
\eeq
satisfy the equations
\beq{abeleq2}
\dot \xi_1=1, \quad \dot \xi_2=\dots =\dot \xi_n=0.
\eeq
\end{prop}

Note that the dynamics of the elliptic coordinates $u_1,\dots, u_n$ is described by the {\it Dubrovin equations} \cite{Dub}:
\beq{Dub}
\dot u_i=\frac{\sqrt{R(u_i)}}{\prod_{j\neq i}(u_i-u_j)}.
\eeq

So the problem is reduced now to the inversion of the following Abel map:
$$\xi_k=\sum_{j=1}^n \int^{u_j} \frac{z^{n-k}}{\sqrt{R(z)}}dz,\,\,\,k=1, \dots, n,$$
which is the celebrated Jacobi inversion problem.

When $n=1$ we have the Neumann system on the unit circle. In that case $\dot \xi=1,$ where
$$\xi= \int^{u} \frac{dz}{2\sqrt{R(z)}}, \, R(z)=-B(z)P(z)=-(z-b_0)(z-b_1)(z+2h)$$
with the inversion given by the Weierstrass elliptic $\wp$-function.

%
For $n=2$, which was the original Neumann case, the complete solution of the Hamilton-Jacobi equation is
$$S=\int^{u_1}\sqrt{\frac{z^2+2hz+2c}{-4B(z)}}dz + \int^{u_2}\sqrt{\frac{z^2+2hz+2c}{-4B(z)}}dz.$$

%
%

The dynamics of the elliptic coordinates can now be described by
\begin{equation}\label{ph}
p_h=\frac{\partial S}{\partial h}=\int^{u_1}\frac{zdz}{\sqrt{R(z)}} + \int^{u_2}\frac{zdz}{\sqrt{R(z)}}=t,
\end{equation}
\begin{equation}\label{pc}
p_{c}=\frac{\partial S}{\partial c}=\int^{u_1}\frac{dz}{\sqrt{R(z)}} + \int^{u_2}\frac{dz}{\sqrt{R(z)}}=0
\end{equation}
where $R(z)=-4B(z)(z^2+2hz+2c).$

The set $\Gamma \subset \mathbf C^2$ given by $y^2=R(z)$ is a hyperelliptic Riemann surface of genus 2. The 1-forms
$$\omega_1=\frac{dz}{\sqrt{R(z)}}=\frac{dz}{y}, \quad \omega_2= \frac{zdz}{\sqrt{R(z)}}=\frac{zdz}{y}$$ do not have singularities on $\Gamma$
(Abelian differentials of the first kind)  and form a linear basis in the space of all such differentials.

The Abel map $\mathcal A: S^2(\Gamma) \rightarrow J(\Gamma)$ is defined by
$$\xi_1=\int^{P_1}\omega_1 + \int^{P_2}\omega_1, \,\, \xi_2=\int^{P_1}\omega_2 + \int^{P_2}\omega_2, \, P_1, P_2 \in \Gamma$$
where $J(\Gamma) = \mathbf C^2/\mathcal L$  with the lattice $\mathcal L$ is generated by the periods of $\omega_1, \omega_2,$ is the Jacobi variety of $\Gamma$.
The Jacobi inversion problem in this genus $g=2$ case was solved by in 1846-47 by G\"opel and Rosenhain.

In the general case it was done by Riemann in 1857, who introduced the Riemann theta functions for any algebraic curve.
Using these functions one can write down the solutions of Neumann system in all dimensions by the explicit formulas, but we will use instead the equations of motion in terms of the abelian integrals (\ref{abeleq1}),(\ref{abeleq2}).

For this we will need the following relation of the separation method, used by Neumann, with the Uhlenbeck-Devaney integrals, which was found by Moser \cite{M1}.

Consider the function
\begin{equation}\label{eq psi}
\Psi_u(p,q)=(1+(R_up,p))(R_uq,q)-(R_u p,q)^2, \quad R_u=(uI-B)^{-1}.
\end{equation}
Geometrically the condition $\Psi_u(p,q)=0$ with given $q$ determines the cone of tangents to the quadric given by $(R_uq,q)+1=0.$

We have the simple fraction expansion
\begin{equation}\label{eq psi2}
\Psi_u(p,q)=\sum_{k=0}^n \frac{F_k(p,q)}{u-b_k},
\end{equation}
where
\begin{equation}\label{eq intN}
F_k(p,q)=q_k^2+\sum_{j\neq k}\frac{(p_kq_j-p_jq_k)^2}{b_k-b_j}, \quad k=0,1,\dots, n
\end{equation}
are the integrals (\ref{eq int}) with $x$ replaced by $p$, $y$ by $q$ (hodograph-like transformation) and $A$ replaced by $B=A^{-1}$ (duality).

\begin{Theorem} (Moser \cite{M1})
The functions $F_k$ given by \eqref{eq intN} are the involutive integrals of the Neumann system \eqref{neu}, satisfying the relations:
$$
\sum_{k=0}^n F_k(p,q)=1, \quad \sum_{k=0}^n b_kF_k=2H.
$$
The function $\Psi_u(p,q)$ can be expressed in terms of the separation parameters (\ref{HJ}) as
\begin{equation}\label{fsep}
\Psi_u(p,q)=\frac{P(u)}{B(u)}=\frac{u^n+2c_1u^{n-1}+ \dots + 2c_n}{\prod_{k=0}^n(u-b_k)}.
\end{equation}
\end{Theorem}

It turned out that the Jacobi and Neumann systems are connected more directly by the Gauss map, as it was discovered by Kn\"orrer \cite{Kn}.

\section{Kn\"orrer's theorem}


Consider as before a geodesic on the hyperboloid $(Bx,x)=1, B=A^{-1}$ satisfying the equations
$$
x''=-\lambda Bx, \,\, \lambda=\frac{(Bx',x')}{(Bx,Bx)}.
$$

Assuming that the Joachimsthal integral $J\neq 0$ denote its sign as $$\varepsilon=J/|J|=\pm 1.$$

Let us change the length parameter $s$ to $\tau$ such that

\beq{tau}
\frac{d\tau}{ds}=\alpha(s), \quad
\alpha^2=|\lambda|=\frac{|(Bx',x')|}{|Bx|^2}=\frac{|J|}{|Bx|^4}.
\eeq

\begin{Theorem} (Kn\"orrer \cite{Kn}) For any reparametrised geodesic $x(\tau)$ on the hyperboloid with $J\neq 0$ the normal vector $q=\frac{Bx}{|Bx|}$ satisfies the equations of Neumann problem on unit sphere $S^n$ with the Hamiltonian
$$
H=\frac{1}{2}|p|^2+\frac{1}{2}\varepsilon (Bq,q), \quad \varepsilon=J/|J|.
$$
The corresponding trajectories $q(\tau)$ satisfy the relation
\begin{equation}\label{psi0}
\Psi^\varepsilon_0(p,q)=0,
\end{equation}
where $\Psi^\varepsilon_0(p,q)$ is given by (\ref{eq psi}) with $B$ replaced by $\varepsilon B$. The generating functions of the corresponding integrals are related by
\begin{equation}\label{psiphi}
\begin{array}{rcl}
|Bx|^4\Psi^\varepsilon_u(p,q)=\Phi_z(x,\dot{x}), \quad u=\frac{\varepsilon}{z}.
\end{array}
\end{equation}
\end{Theorem}

\begin{proof}
We have
 $$x'=\frac{dx}{ds}=\frac{dx}{d\tau}\frac{d\tau}{ds}=\alpha\dot{x}, $$
     $$x''=\frac{dx'}{ds}=\frac{dx'}{d\tau}\frac{d\tau}{ds}=\alpha\dot{\alpha} \dot{x}+\alpha^2 \ddot{x},$$
     where $\dot{}=\frac{d}{d\tau}.$
 Thus, the geodesic equations in Kn\"orrer's parameter are
 $$
 \alpha^2 \ddot x+\alpha \dot \alpha \dot x =-\varepsilon \alpha^2 Bx,
 $$
 or, after division by $\alpha^2$
 \beq{kgeod}
\ddot x+\frac{\dot \alpha}{\alpha} \dot x =-\varepsilon Bx.
 \eeq

Let  $$q=\frac{Bx}{|Bx|}=\gamma Bx,  \,\,\gamma=\frac{1}{|Bx|},$$ then
\begin{equation}\label{eq 2.9}
\begin{array}{rcl}
\ddot{q}&=&\ddot{\gamma}Bx+2\dot{\gamma}B\dot{x}+\gamma B\ddot{x}\\
  &=&\ddot{\gamma}Bx+\gamma B(\ddot{x}+2\frac{\dot{\gamma}}{\gamma} \dot{x}).\\
\end{array}
\end{equation}
Differentiating the relation $\alpha^2\gamma^{-4}=|J|$ with respect to $\tau$, we get
$
2\frac{\dot{\gamma}}{\gamma}-\frac{\dot{\alpha}}{\alpha}=0,
$
and thus
\begin{equation}\label{eq 2.12}
\ddot{q}=\ddot{\gamma}Bx+\gamma B(\ddot{x}+\frac{\dot{\alpha}}{\alpha} \dot{x})=-\varepsilon Bq +\mu q,
\end{equation}
which are the equations of motion of Neumann system with the potential $\frac{1}{2}\varepsilon (Bq,q),$ where $\mu$ is the Lagrange multiplier: $\mu=\varepsilon (Bq,q)-(\dot q, \dot q).$

Since $\Phi_0=0$ due to (\ref{eq phi0}) the relation (\ref{psi0}) now follows from (\ref{psiphi}), which can be checked directly, see \cite{Kn, M2}.
\end{proof}

Let $d_1,\dots, d_{n-1}$ be the non-zero roots of $\Psi^\varepsilon_u(p,q)=0$:
\begin{equation}\label{psi00}
\Psi^\varepsilon_u(p,q)=\frac{u\prod_{i=1}^{n-1}(u-d_i)}{\prod_{k=0}^{n}(u-b_k)}.
\end{equation}
Then the corresponding spectral curves of Neumann and Jacobi systems
$$
y^2=\Psi^\varepsilon_u(p,q), \quad w^2=\Phi_z(x,y),
$$
are related by the change
$u=\varepsilon z^{-1}.$

In particular, if $J>0$ then $b_k=a_k^{-1}, \, d_i=c_i^{-1}$ and $u_i=z_i^{-1}$, where $z_1,\dots, z_n$ are the corresponding Jacobi elliptic coordinates on the quadric defined as the non-zero roots of
\beq{elk}
\sum_{i=0}^n \frac{x_i^2}{z-a_i}+1=0.
\eeq
For $J<0$ we have respectively $b_k=-a_k^{-1}, \, d_i=-c_i^{-1}, \,u_i=-z_i^{-1}.$

\section{Geodesic scattering}

We are ready now to describe the geodesic scattering on the hyperboloids.
First of all let us look more closely at the change of parametrisation $s \to \tau$ given by
$$
\frac{d\tau}{ds}=\frac{\sqrt{|J|}}{|Bx|^2}.
$$
Since $|Bx|$ grows linearly at infinity we see that $\lambda \approx C s^{-4}$ and thus
$\frac{d\tau}{ds}$ decays as $s^{-2}$. This means that the integral $\tau = \int \alpha(s)ds$
is convergent at infinity, so it takes finite time in the new time variable $\tau$ to reach infinity (cf. \cite{V1}).

The Kn\"orrer result explains how to extend this dynamics beyond this point. Note that the two "infinities" of the hyperboloid under the Gauss map correspond to two components $C_{\pm}$ of the intersection of the unit sphere $|q|^2=1$ with the cone $(Aq,q)=0,$ which is dual to the asymptotic cone of the hyperboloid $(A^{-1}x,x)=1.$ We will call these hypersurfaces $C_{\pm}\subset S^n$ {\it asymptotic}.

\begin{prop}
The image of a geodesic on the hyperboloid corresponds to the finite part of corresponding trajectory $q(\tau)$ of the Neumann system between two consecutive points of its contact with the asymptotic hypersurfaces $C_{\pm}.$ At the contact points trajectory is either tangent to the hypersurfaces, or $\dot q=0.$
\end{prop}

\begin{proof}
We have only to prove the last statement. Due to (\ref{psi0})
$$\Psi^\varepsilon_0(p,q)=(-\varepsilon+(Ap,p))(Aq,q)-(Ap,q)^2=0,$$
so when $(Aq,q)=0$ we have $(Ap,q)=(Aq,\dot q)=0.$ This means that the trajectory is tangent to the hypersurface unless $\dot q=0.$
\end{proof}

In the $n=1$ we have only two geodesics corresponding to two branches of hyperbola
$$b_0x_0^2 + b_1x_1^2=1,$$
with $b_0<0<b_1$. On the Neumann side we have the system equivalent to the mathematical pendulum. Indeed, in the angle coordinate $\phi$ the potential takes the form
$-\frac{1}{2}(b_0 q_0^2 + b_1 q_1^2)=-\frac{1}{2}(b_0 cos^2\phi + b_1\sin^2 \phi),$ which is up to a constant proportional to $\cos 2\phi.$
So on this side we have periodic back and forth oscillations between two points of intersection of the unit circle with two lines $a_0 q_0^2 + a_1 q_1^2=0,$ where $\dot q=0.$
The geodesics correspond to half-periods of these oscillations.

Let us discuss now the $n=2$ case, which is much more instructive.
In that case we have
$$
\Phi_z(x,y)=\frac{F_0(x,y)}{z-a_0}+\frac{F_1(x,y)}{z-a_1}+\frac{F_2(x,y)}{z-a_2}=\frac{z(z-c)}{(z-a_0)(z-a_1)(z-a_2)},
$$
where as before $$
F_k(x,y)=y_k^2+\sum_{j\neq k}\frac{(x_ky_j-x_jy_k)^2}{a_k-a_j}, \quad k=0, 1, 2,
$$
are Uhlenbeck-Devaney integrals, and due to (\ref{JF}) $c$ can be expressed via Joachimsthal integral by $c=a_0a_1a_2 J.$
Note that since $a_2-a_j>0$ for $j=0,1$ the integral $F_2$ cannot be negative.
A simple analysis of the possible graphs of $\Phi_z$ (see Fig. 2) shows that there are 4 different major cases of positions of constant $c$ and the Jacobi elliptic coordinates $z_1<z_2$, depending on the signs of the integrals $F_0, F_1$ and $J$:
$$ \text{(I)} \quad F_0<0, \, F_1<0, \,\, J>0: \quad c<a_0, \, z_1<c, \, a_1<z_2<a_2;\quad \quad$$
$$ \text{(II)} \quad F_0>0, \, F_1<0, \,\, J>0: \quad a_0<c<0,\, z_1<0, \, a_1<z_2<a_2;$$
$$ \text{(III)} \quad F_0>0, \, F_1<0, \,\, J<0: \quad 0<c<a_1,\, z_1<0, \, a_1<z_2<a_2;$$
$$ \text{(IV)} \quad F_0>0, \, F_1>0, \,\,  J<0: \quad a_1<c<a_2, \, z_1<0, \, c<z_2<a_2.$$
Here we have used the interpretation of $J$ as the derivative of $\Phi_z$ at $z=0$, see (\ref{JF}).

\begin{figure}[h]
\centering
\parbox[c]{0.48\linewidth}
{\includegraphics[width=\linewidth]{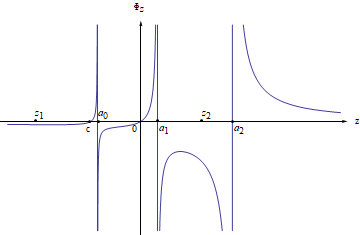}\\ \centering {\footnotesize{$(I)\ F_0<0,F_1<0,J>0$}}}
  \parbox[c]{0.48\linewidth}
{\includegraphics[width=\linewidth]{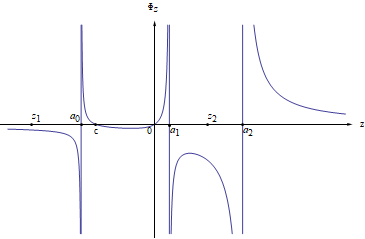}\\ \centering{\footnotesize{$(II)\ F_0>0,F_1<0,J>0$}}}
\parbox[c]{0.48\linewidth}
{\includegraphics[width=\linewidth]{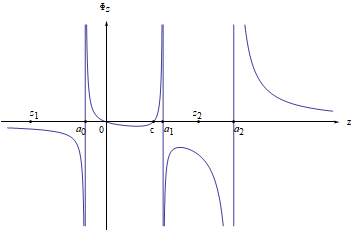}\\ \centering{\footnotesize{$(III)\ F_0>0,F_1<0,J<0$}}}
 \parbox[c]{0.48\linewidth}
{\includegraphics[width=\linewidth]{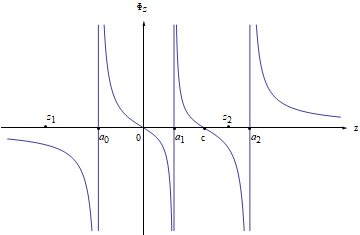}\\ \centering{\footnotesize{$(IV)\ F_0>0,F_1>0,J<0$}}}
\caption{Possible graphs of $\Phi_z$}
\end{figure}

Now we can answer our first question about the geodesic scattering on the hyperboloids.

\begin{Theorem}
A geodesic on one-sheeted hyperboloid $$\frac{x_0^2}{a_0} + \frac{x_1^2}{a_1} +\frac{x_2^2}{a_2} =1$$ with $a_0<0<a_1<a_2$ is crossing the neck ellipse with $x_0=0$  if and only if the integral
\beq{F0}
F_0=y_0^2+\sum_{j=1}^2\frac{(x_0y_j-x_jy_0)^2}{a_0-a_j}>0.
\eeq
When $F_0<0$ the geodesic is reflected back to the infinity. The geodesics with $F_0=0$ are approaching the neck geodesic, spiralling around the hyperboloid.
\end{Theorem}

\begin{proof}
Indeed, the neck ellipse $x_0=0$ in the elliptic coordinates corresponds to $z_1=a_0$. We can see from above analysis that $z_1$ always hits $a_0$ except the case (I) with $F_0<0.$
In the case (I) we have $z_1\leq c<0$, so the geodesic is reflected back to infinity.

In the critical case $F_0=0$ we have $F_1+F_2=|y|^2$ and $a_1^{-1}F_1+a_2^{-1}F_2=0,$ which implies that $F_1<0<F_2.$
In this case the polynomial $$R(z)=-4z(z-a_0)^2(z-a_1)(z-a_2)$$ has a double root and the corresponding spectral curve is singular.

On the Neumann side we have the Dubrovin equations
$$\dot u_1=\frac{\sqrt{R(u_1)}}{u_1-u_2},\quad \dot u_2=\frac{\sqrt{R(u_2)}}{u_2-u_1}$$
with $R(u)=-4u(u-b_0)^2(u-b_1)(u-b_2).$ We see that when $u_1$ is close to $b_0$ the derivative $\dot u_1$ is proportional to $b_0-u_1$, which means that $u_1-b_0 \approx e^{-c_0\tau}.$
Since $z_1=u_1^{-1}$ we see that $z_1-a_0$ also decays exponentially, which implies the claim.
\end{proof}

The degenerate solutions of the Neumann system were studied by Moser in relation with the soliton solutions of the Korteweg--de Vries equation \cite{M3}.
In our case they correspond to the one-solitons at the background of one-gap potentials first discussed by Kuznetsov and Mikhailov \cite{KM} (the general $n$-gap case was studied by Krichever \cite{Krich}).

Similar arguments lead to the general result stated in the Introduction (Theorem 1.1).

Let us now address Question 2, using Kn\"orrer's identification of the Jacobi and Neumann system.

Consider the asymptotic hypersurfaces $C_{\pm}\subset  S^N$ given by $$(Aq,q)=0, |q|^2=1.$$
In the Neumann elliptic coordinates $u_1,\dots, u_n$ which are the solutions of
$$
((u-\varepsilon B)^{-1}q,q)=0,
$$
they are simply the coordinate surfaces $u_n=0$ (we choose now for convenience the ordering $u_1>u_2>\dots >u_n$). It is natural therefore to use the remaining $u_1,\dots, u_{n-1}$
as the coordinates on $C.$ There is a problem though, since knowing the relation (\ref{qii}) defines $q_i$ only up to a sign.

To deal with this issue we consider the corresponding {\it real}
spectral curve $\Gamma$ given by
$$
y^2=R(u):=-4u\prod_{i=1}^{n-1}(u-d_i)\prod_{k=0}^{n}(u-b_k).
$$
It consists of one unbounded component and $n$ ovals $\alpha_1, \dots, \alpha_n$, one of them (which we choose to be $\alpha_n$) containing point $(0,0)$.

Note that when $u=u_j$ then $(R_uq,q)=((u_j-\varepsilon B)^{-1}q,q)=0$, and thus $\Psi^\varepsilon_{u_j}=-(R_{u_j} p,q)^2$. Hence if we know both $p$ and $q$ this allows us to define $\sqrt {R(u_j)}$ uniquely as
$$
\sqrt {R(u_j)}=2(R_{u_j} p,q)\prod_{k=0}^{n}(u_j-b_k)
$$
and as a result well-defined point $P_j=(u_j, \sqrt{R(u_j)}) \in \Gamma$ (see Moser \cite{M2}).
The projection of the dynamics on the $u$ coordinate is described by the Dubrovin equations (\ref{Dub}), which
determine the orientation on the ovals $\alpha_1, \dots, \alpha_n$.

The point $P_j$ is rotating along the corresponding oval $\alpha_j$, and when it passes through the branching points $(b_j,0)$ the corresponding coordinate $q_j$ changes sign.
In this way we control the sign of the coordinates as well.

Let $D=(P_1,\dots, P_{n-1})$ be the corresponding set of points on $\Gamma$ and define a real version of (partial) Abel map, which is a diffeomorphism
\begin{equation}\label{AD1}\mathbb A:\alpha_1\times\dots\times \alpha_{n-1} \to \mathbb T^{n-1}=\mathbb R^{n-1}/\mathcal L
\end{equation}
defined by
\begin{equation}\label{AD2}
{\mathbb A}(D)_j:=\sum_{i=1}^{n-1}\int^{P_i} \omega_j, \quad \omega_j=\frac{u^{j-1}du}{\sqrt{R(u)}}, \quad j=1,\dots, n-1,
\end{equation}
and the lattice $\mathcal L$ is generated by the period vectors $\Omega_1,\dots, \Omega_{n-1}$:
\begin{equation}\label{AD3}
\Omega_j=(\oint_{\alpha_1}\omega_j,\dots, \oint_{\alpha_{n-1}}\omega_j), \quad j=1,\dots, n-1.
\end{equation}

Note that we do not use here the point $P_n$ on the remaining oval $\alpha_n$ and the last holomorphic differential $\omega_n=u^{n-1}du/\sqrt{R(u)}$. This map works well {\it only on the reals}, see the discussion of the complex situation in \cite{V0} in relation with the billiard on the ellipsoid.

Now the we can reformulate the scattering problem as the relation between the divisors $D_-$ and $D_+$ corresponding to the limits of the geodesic as $s\to \pm\infty.$

\begin{Theorem}\label{scattering}
The geodesic scattering on the hyperboloid is described by the relation
\begin{equation}\label{scat}
{\mathbb A}(D_+)={\mathbb A}(D_-)-\Delta, \quad \Delta_j=\oint_{\alpha_n} \omega_j,  \quad j=1,\dots, n-1.
\end{equation}
The topological properties of the geodesic are completely determined by the position of the vector $\Delta$ with respect to the lattice $\mathcal L.$
\end{Theorem}

\begin{proof}
From the Proposition (\ref{propos}) we see that along the trajectory $\xi_2,\dots, \xi_n$ are preserved, which means that
\beq{xi0}
\sum_{j=1}^n \int^{P_j} \frac{z^{n-k}}{\sqrt{R(z)}}dz=\xi^0_k, \quad k=2,\dots n.
\eeq
This means that
$$
{\mathbb A}(D_+)+{\mathbb A}(P_n^+)={\mathbb A}(D_-)+{\mathbb A}(P_n^-),
$$
where
$$
{\mathbb A}(P)_j=\int^P \omega_j, \quad j=1,\dots, n-1.
$$
As we know the geodesic corresponds to the part of the Neumann trajectory described by $P_n$ travelling exactly once along the oriented oval $\alpha_n$, the difference
$$
{\mathbb A}(P_n^+)-{\mathbb A}(P_n^-)=\Delta.
$$
This proves the relation (\ref{scat}), which determines uniquely the map $D_-\to D_+$ and hence by projection to $u$ coordinates the change of the coordinates $u_1,\dots, u_{n-1}$ on $C.$ The shift vector $\Delta$ contains all the information we need to recover the signs of the corresponding coordinates $q_0,\dots, q_n$ and the topological properties of the trajectory.
\end{proof}

Let us demonstrate this in the $n=2$ case assuming for simplicity the reflection case (I). The position of roots and ovals in this case is shown on Fig. 3.

\begin{figure}[h]
  \includegraphics[width=90mm]{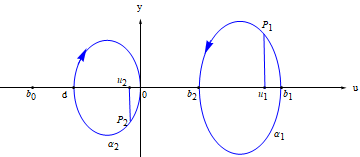}
  \caption{The ovals and position of roots in the reflection case}
\end{figure}
In that case the asymptotic curve $C$ can be identified with the double cover of the oval $\alpha_1$: when $P_1$ goes once around the oval, it passes through once through branching points $(b_1,0)$ and $(b_2,0)$, which changes the sign of the corresponding coordinates $q_1$ and $q_2$ respectively.

In particular, we see that the corresponding geodesic makes
\beq{N}
N=\left[ \frac{I_2}{2I_1} \right],
\eeq
full rotations about the hyperboloid before it is reflected back to infinity, where
$$
I_1=\oint_d^0\frac{du}{\sqrt{R(u)}}, \quad I_2=\oint_{b_1}^{b_2}\frac{du}{\sqrt{R(u)}}.
$$
The analysis in the general case is similar.

\section{Two-sheeted hyperboloid and cone cases}

Let us briefly discuss what is happening on the two-sheeted hyperboloid
$
Q(x)=(A^{-1}x,x)=(Bx,x)=-1
$
and cone $Q(x)=0$ with the same assumption
$a_0<0<a_1<\dots<a_n$ (see two-dimensional case on Fig.4).

\begin{figure}[h]
  \includegraphics[width=35mm]{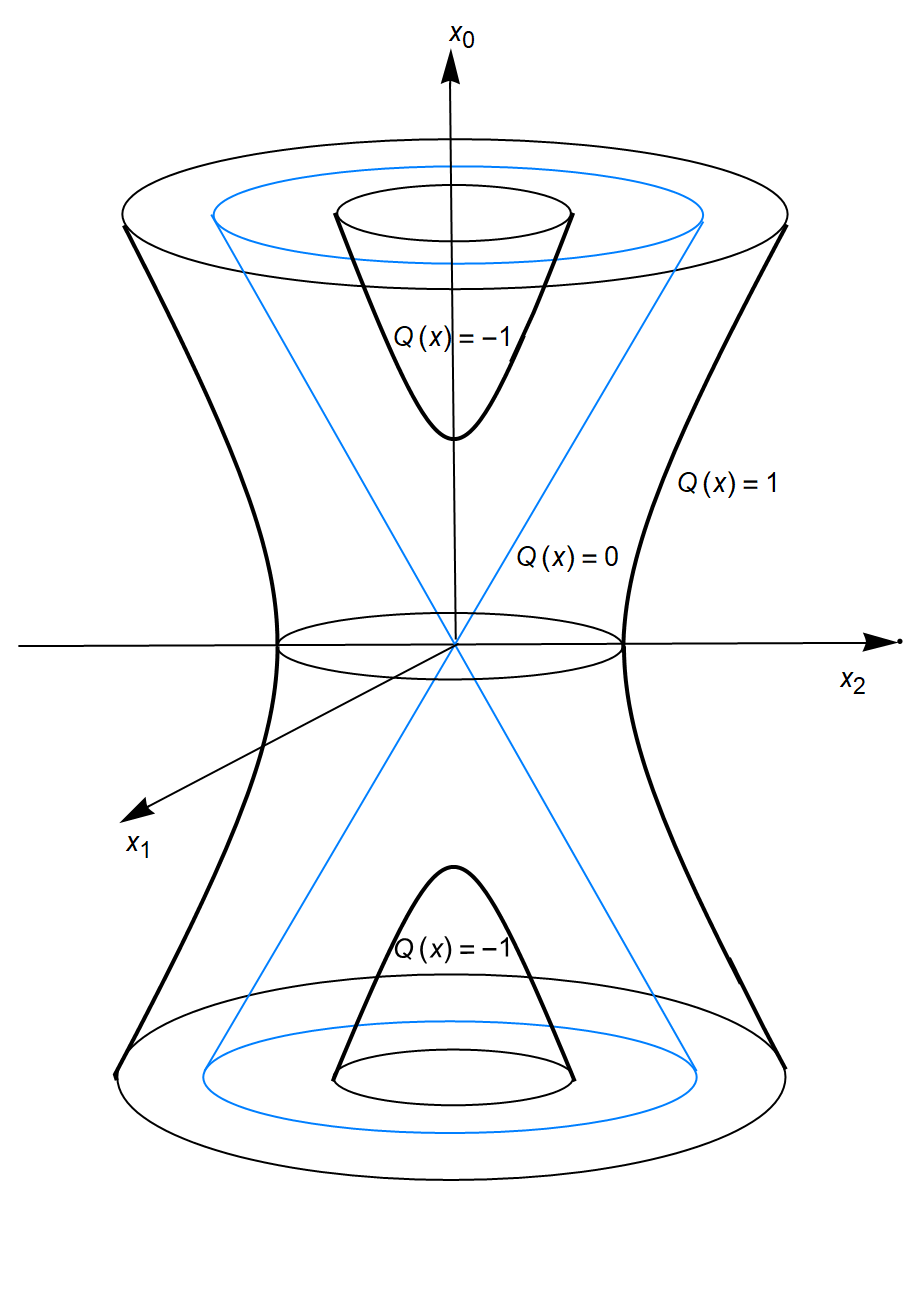}
  \caption{Hyperboloids and cone in two dimensions}
\end{figure}

The geodesics in the natural parameter $s$ are given by the same equation
$$
x^{''}=-\lambda Bx, \,\,\,\, \lambda=\frac{(Bx',x')}{(Bx,Bx)}=\frac{ J}{|Bx|^4},
$$
but now in two-sheeted hyperboloid case $J=(Bx,Bx)(Bx',x')$ is always positive.

Consider the following version of (\ref{eq phi})
$$
\Phi_z(x,y)=((R_zx,x)-1)(R_zy,y)-(R_z x,y)^2, \quad R_z=(zI- A)^{-1},
$$
which can be written as
$$
\Phi_z(x,y)=\sum_{k=0}^n \frac{F_k(x,y)}{z- a_k}=-\frac{z\prod_{i=1}^{n-1}(z-c_i)}{\prod_{k=0}^{n}(z- a_k)},
$$
$$
F_k(x,y)=-y_k^2+\sum_{j\neq k}\frac{(x_ky_j-x_jy_k)^2}{a_k-a_j}, \quad k=0, 1,\dots, n.
$$
Note that $F_0(x,y)$ is always negative, meaning that all geodesics are reflected back, which is, of course, obvious by geometric reasons.
 It is easy to show also that
$$
\sum_{k=0}^n a_k^{-1}F_k=0, \quad \sum_{k=0}^n F_k=- |y|^2,
$$
$$
  J=\frac{d}{dz}\Phi_z(x,y)|_{z=0}=-\frac{\prod_{i=1}^{n-1}c_i}{\prod_{k=0}^{n} a_k}=-\sum_{k=0}^n a_k^{-2}F_k.
$$
The elliptic coordinates $z_1,\dots, z_n$ defined as the non-zero roots of the equation
$$
((zI-A)^{-1}x,x)=1
$$
and after some reparametrisation the dynamics becomes linear at the Jacobi variety of the corresponding hyperelliptic curve
$$
w^2+R(z)=0, \quad R(z)=z\prod_{i=1}^{n-1}(z-c_i)\prod_{k=0}^{n}(z-a_k).
$$

Since $J>0$ in this case we can change the length parameter $s$ to $\tau$ like in ellipsoid case
$$
\frac{d\tau}{ds}=\alpha(s), \quad
\alpha^2=\lambda=\frac{(Bx',x')}{|Bx|^2}=\frac{J}{|Bx|^4}.
$$

From Kn\"orrer \cite{Kn} we have that for any reparametrised geodesic $x(\tau)$ the normal vector $q=\frac{Bx}{|Bx|}$ satisfies the equations of Neumann problem on unit sphere $S^n$
$$\ddot{q}=- Bq +\mu q.$$
The corresponding trajectories $q(\tau)$ satisfy the relation $\Psi_0(p,q)=0,$ where
$\Psi_u(p,q)$ is the generating function of the Neumann integrals given by (\ref{eq psi}).

Let $d_1,\dots, d_{n-1}$ be the non-zero roots of $\Psi_u(p,q)=0$:
$$
\Psi_u(p,q)=\frac{u\prod_{i=1}^{n-1}(u-d_i)}{\prod_{k=0}^{n}(u-b_k)}.
$$
Then the corresponding spectral curves of Neumann and Jacobi systems
$$
y^2=\Psi_u(p,q), \quad w^2=\Phi_z(x,y),
$$
are related by the change
$u=z^{-1}.$ Hence $b_k=a_k^{-1}, \, d_i=c_i^{-1}$ and $u_i=z_i^{-1}$, where $z_1,\dots, z_n$ are the corresponding elliptic coordinates on hyperboloid.

Let us discuss now the two-dimensional case. When $n=2$ we have
$$
\Phi_z(x,y)=\frac{F_0(x,y)}{z-a_0}+\frac{F_1(x,y)}{z- a_1}+\frac{F_2(x,y)}{z- a_2}=\frac{-z(z-c)}{(z- a_0)(z- a_1)(z- a_2)},
$$
$$
F_k(x,y)=-y_k^2+\sum_{j\neq k}\frac{(x_ky_j-x_jy_k)^2}{a_k-a_j}, \quad k=0, 1, 2,
$$
$$
\sum_{k=0}^2 a_k^{-1}F_k=0, \quad \sum_{k=0}^2 F_k=-1,\quad c=-a_0a_1a_2 J>0.$$
As we have already mentioned, we have $F_0<0$ which means that all the geodesics are reflected back.

Depending on the signs of the integrals $F_1, F_2$ we have the following two possibilities for positions of $c$ and the elliptic coordinates $z_1<z_2$:
$$ \quad\quad (I)\,\, F_1<0, \,\, F_2<0: \quad a_1<c<a_2, \quad a_1<z_1<c,\quad a_2<z_2,$$
$$ (II) \,\,  F_1<0, \,\, F_2>0: \quad a_2<c, \quad  a_1<z_1<a_2,\quad c<z_2,$$
or, in terms of Neumann parameters:
$$ (I)\,\,b_0<0<u_2<b_2<d<u_1<b_1;$$
$$ (II)\,\,b_0<0<u_2<d<b_2<u_1<b_1.$$

In particular, in case (I) we see that the corresponding geodesic makes
$$
N=\left[ \frac{I_2}{2I_1} \right],
$$
full rotations about the hyperboloid before it is reflected back to infinity, where
$$
I_1=\oint_{0}^{b_2}\frac{du}{\sqrt{R(u)}}, \quad I_2=\oint_{b_1}^{d}\frac{du}{\sqrt{R(u)}}, \quad R(u)=-4u(u-d)(u-b_0)(u-b_1)(u-b_2).
$$

In the case of cone given by
$$Q(x)=\frac{x_0^2}{a_0}+\frac{x_1^2}{a_1}+\frac{x_2^2}{a_2}=0$$
the metric becomes flat. Flattening cone to the plane we have a corner with angle $\alpha$ given by the length of the intersection of the cone with the unit sphere $|x|^2=1.$

To find $\alpha$ differentiate the corresponding equations  $$\frac{x_0^2}{a_0}+\frac{x_1^2}{a_1}+\frac{x_2^2}{a_2}=0,\quad  x_0^2+x_1^2+x_2^2=1,$$
to get
$$\frac{x_0dx_0}{a_0}+\frac{x_1dx_1}{a_1}+\frac{x_2dx_2}{a_2}=0, \quad x_0dx_0+x_1dx_1+x_2dx_2=0,$$
which implies that
$$ds=\sqrt{dx_0^2+dx_1^2+dx_2^2}=\sqrt{\frac{1-k_1k_2x_2^2}{(1-k_1x_2^2)(1-k_2x_2^2)}}dx_2$$
with
$$0<k_1=\frac{a_2-a_1}{a_2}<1, \quad k_2=\frac{a_2-a_0}{a_2}>1. $$
Introducing $t=\sqrt{k_2}x_2$ we have the following complete elliptic integral
\beq{alpha}
\alpha=\frac{4}{\sqrt{k_2}} \int_0^{1} \sqrt{\frac{1-k_1t^2}{(1-t^2)(1-\frac{k_1}{k_2}t^2)}}dt.
\eeq
In particular, in symmetric case $a_1=a_2$ we have $k_1=0$ and
$$\alpha=\frac{4}{\sqrt{k_2}}  \int_0^{1} \frac{dt}{\sqrt{1-t^2}}=2\pi \sqrt{\frac{a_2}{a_2-a_0}}.$$
The geodesics will become the billiard trajectories in the corresponding corner, see Fig. 5.

\begin{figure}[h]
  \includegraphics[width=70mm]{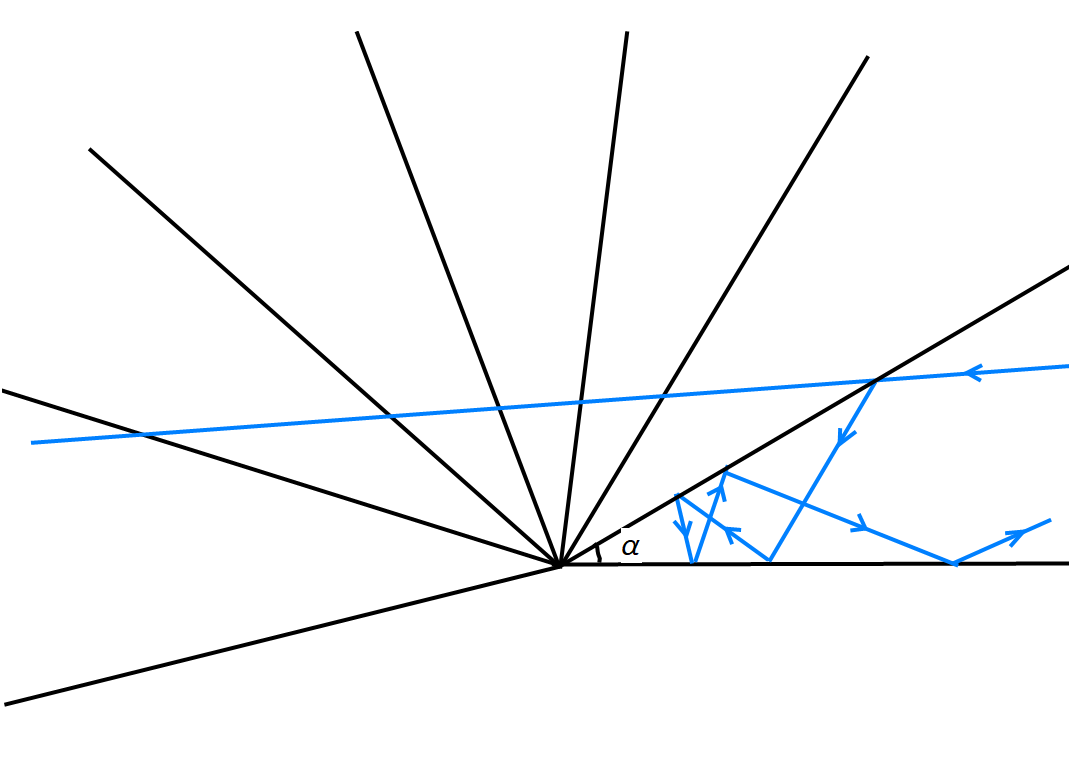}
  \caption{Geodesic scattering on cone and billiard in the corner}
\end{figure}

As one can see from this figure, the corresponding scattering is the shift by $\pi$ on the circle $\mathbb R/\alpha \mathbb Z$ with $\alpha$ given by (\ref{alpha}).
One can check that this is a limiting case of the general formula (\ref{scat}) above.

\section{Symmetric two-dimensional case}

Consider now the case of the rotationally-invariant one-sheeted hyperboloid in $\mathbb{R}^{3}$
\begin{equation}\label{eq 6.1}
\frac{x^2+y^2}{a^2}-\frac{z^2}{c^2}=1.
\end{equation}
Choose the following parametrization
\begin{equation}\label{eq 6.2}
\begin{array}{rcl}
\left\{
\begin{array}{lcl}
x=a\sqrt{1+\rho^2}\cos\phi,\\
y=a\sqrt{1+\rho^2}\sin\phi,\\
z=c\rho,
\end{array}
\right.
\end{array}
\end{equation}
with $\rho\in \mathbb R,\, \phi\in\mathbb R \,(\text{mod} \, 2\pi).$

The metric on the hyperboloid in these coordinates has the form
$$
ds^2=\frac{c^2+b^2\rho^2}{1+\rho^2} d\rho^2+a^2(1+\rho^2)d\phi^2, \quad b^2:=a^2+c^2,
$$
so the corresponding Hamiltonian is
\beq{HS}
H=\frac{1}{2}\left [\frac{c^2+b^2\rho^2}{1+\rho^2}\dot \rho^2+a^2(1+\rho^2)\dot \phi^2\right]=\frac{1}{2}\left [\frac{1+\rho^2}{c^2+b^2\rho^2}p_\rho^2+\frac{p_\phi^2}{a^2(1+\rho^2)}\right].
\eeq

Since $H$ does not depend on $\phi$
we have the integral
\begin{equation}\label{I}
I=p_\phi=\frac{\partial H}{\partial \dot \phi}=a^2(1+\rho^2) \dot \phi, \quad \frac{dI}{dt}=0.
\end{equation}

Let $\alpha$ be the angle of the geodesic with the parallels defined by $\rho=const$, then an easy calculation shows that
$$
\cos \alpha=\frac{a\sqrt{1+\rho^2}\dot\phi}{\sqrt{2H}}.
$$
The classical Clairaut integral for the geodesics on the surfaces of revolution, which is defined in our case as
$$
C:=a\sqrt{1+\rho^2}\cos \alpha,
$$
can be expressed then in terms of $I$ and $H$ as
$$
C=a\sqrt{1+\rho^2}\cos \alpha=\frac{a^2(1+\rho^2)\dot\phi}{\sqrt{2H}}=\frac{I}{\sqrt{2H}}.
$$
For given value of $C$ we have the inequality
$$\frac{C^2}{a^2(1+\rho^2)}=\cos \alpha^2 \leq 1,$$
or, equivalently
$$
\rho^2\geq \frac{C^2}{a^2}-1=\frac{I^2}{2a^2H}-1.
$$
In particular, in the length parametrisation $t=s$ we have $2H=1$ and
$C^2=I^2.$

This implies the following

\begin{prop}
A geodesic is reflected if and only if the value of the Clairaut integral $C^2>a^2.$
The geodesics with $C^2<a^2$ cross the neck and then go to the opposite infinity.
At the critical level $C^2=a^2$ we have
the neck periodic geodesic and the geodesics which are exponentially approaching it rotating around the hyperboloid.

The Joachimsthal integral $J$ and the corresponding integral (\ref{F0}) have the form
\beq{JHCF}
J=\frac{1}{a^4c^2}\left(\frac{b^2}{a^2}I^2-2a^2H\right), \quad
F_0=\frac{2c^2H}{a^2b^2}(a^2-C^2).
\eeq
\end{prop}

\begin{proof}
Indeed, if $C^2>a^2$ then the geodesic is reflected since $\rho^2\geq \frac{C^2}{a^2}-1>0.$
If $C^2<a^2$ then $\rho$ can take any value, including 0, so the geodesic is crossing the neck circle.

A direct calculation shows that the integrals $J$ and $F_0$ given by (\ref{F0}) with $a_1=a_2=a^2, \, a_0=-c^2$ have the form (\ref{JHCF}),
which, in particular, shows the agreement with the general result from the previous section.

At the critical level $C^2=a^2$ we have
$$
2a^2H-I^2=\frac{a^2(1+\rho^2)}{c^2+b^2\rho^2}p_\rho^2-\frac{\rho^2}{1+\rho^2}p_\phi^2=\frac{a^2(c^2+b^2\rho^2)}{1+\rho^2}\dot\rho^2-\frac{I^2 \rho^2}{1+\rho^2}=0.
$$
The solution $\rho\equiv 0$ corresponds to the periodic neck geodesic. If $\rho \neq 0$ then in the natural parameter $s$ we have $I^2=C^2=a^2$ and
$$
\int\frac{\pm\sqrt{c^2+b^2\rho^2}}{\rho}d\rho=\int ds=s-s_0.
$$
In particular, when $\rho$ is close to zero, then we have the asymptotic behaviour
$$
c \int \frac{d\rho}{\rho}=c \log \rho=-(s-s_0),
$$
which implies the exponential decay
$$
\rho=c_0 e^{-c_1s}, \, c_1= c^{-1}.
$$
The corresponding angular coordinate $\phi$ is defined by
$$
a^2(1+\rho^2) \dot \phi=I=a^2,
$$
which as $\rho \to 0$ asymptotically gives the uniform rotation $\phi=s + const.$
\end{proof}

One can check that the case $J=0$ indeed corresponds to the straight line geodesics.

The scattering and the topological properties of the geodesics are described by the following

\begin{prop}
The total scattering change $\Delta \phi=\phi_\infty - \phi_{-\infty}$ of the angle $\phi$ along a geodesic can be given as
\beq{Delta1}
\Delta \phi=\frac{I}{a}\int_{-\infty}^\infty \sqrt{ \frac{c^2+b^2\rho^2}{(a^2-I^2)+a^2\rho^2}}\frac{d\rho}{1+\rho^2}=\frac{I}{a}\int_{-\pi/2}^{\pi/2} \sqrt{ \frac{c^2+b^2\tan^2 u}{(a^2-I^2)+a^2\tan^2 u}}du
\eeq
when $I^2<a^2$ (transmission case), and
\beq{Delta2}
\Delta \phi=\frac{2I}{a}\int_{\rho_*}^\infty \sqrt{ \frac{c^2+b^2\rho^2}{(a^2-I^2)+a^2\rho^2}}\frac{d\rho}{1+\rho^2}=\frac{2I}{a}\int_{u_*}^{\pi/2} \sqrt{ \frac{c^2+b^2\tan^2 u}{(a^2-I^2)+a^2\tan^2 u}}du,
\eeq
where $\rho_*^2=\tan^2 u_*:=\frac{I^2-a^2}{a^2}$ with $I^2>a^2$ (reflection case).
\end{prop}

\begin{proof}
The dependence of $\phi$ on the length parameter $s$ is defined by (\ref{I}):
$$
\frac{d\phi}{ds}=\frac{I}{a^2}\frac{1}{1+\rho^2}.
$$
From (\ref{HS}) we have
$$
\left(\frac{d\rho}{ds}\right)^2=\frac{(2a^2H-I^2)+2a^2H\rho^2}{a^2(c^2+b^2\rho^2)}=\frac{(a^2-I^2)+a^2\rho^2}{a^2(c^2+b^2\rho^2)}
$$
where we have used that in the length parametrisation $2H=1.$
Thus in the transmission case with $a^2>I^2$ the total scattering change of the angle is determined by the integral
$$
\Delta \phi =\int d\phi=\frac{I}{a^2}\int \frac{ds}{1+\rho^2}
=\frac{I}{a}\int_{-\infty}^\infty\sqrt{\frac{c^2+b^2\rho^2}{(a^2-I^2)+a^2\rho^2}}\frac{d\rho}{1+\rho^2}.
$$
In the reflection case we know that $\rho\geq \rho_*=\sqrt{\frac{I^2-a^2}{a^2}}$ so
$$
\Delta \phi
=\frac{2I}{a}\int_{\rho_*}^\infty\sqrt{\frac{c^2+b^2\rho^2}{(a^2-I^2)+a^2\rho^2}}\frac{d\rho}{1+\rho^2}.
$$
Making the change of variable $\rho=\tan u,$ we come to (\ref{Delta1}), (\ref{Delta2}).
\end{proof}

\section{Quantum case}

Let us briefly discuss the quantum situation in the same simplest case. The corresponding quantum Hamiltonian is the Laplace-Beltrami operator on the hyperboloid (\ref{eq 6.1}) having in the same coordinates the form
$$
\hat H=-\frac{1}{2}\left[\frac{1+\rho^2}{c^2+b^2\rho^2}\frac{\partial^2}{\partial \rho ^2}+\frac{1}{a^2(1+\rho^2)}\frac{\partial^2}{\partial \phi ^2}\right].
$$
Separation of variables $\psi(\rho,\phi)=\psi_1(\rho)\psi_2(\phi)$ in the corresponding stationary Schr\"odinger equation $\hat H\psi=E\psi$ gives
$$
\psi_2''(\phi)=-k^2\psi_2(\phi), \,\, \psi_2=A_k\cos k\phi+B_k \sin k\phi, \quad k \in \mathbb Z_{\geq 0}
$$
and
\beq{Lk}
\mathcal L_k\psi_1(\rho)=E\psi_1(\rho), \quad \mathcal L_k=\frac{1}{2}\left[-\frac{1+\rho^2}{c^2+b^2\rho^2}\frac{d^2}{d\rho^2}+\frac{k^2}{a^2(1+\rho^2)}\right].
\eeq

\begin{prop}
Under the Liouville transformation
\beq{Lt}
\frac{d\tau}{d\rho}=\alpha, \ \ \psi_1(\rho)=\alpha^{-\frac{1}{2}}\varphi(\tau), \ \alpha^2=\frac{2(c^2+b^2\rho^2)}{1+\rho^2},
\eeq
the equation \eqref{Lk} becomes the Schr\"odinger equation
\beq{sh}
\left(-\frac{d^2}{d\tau^2}+V(\tau)\right)\varphi=E\varphi,
\eeq
where
\beq{pot}
V(\tau)=\frac{k^2}{2a^2(1+\rho^2)}-\frac{a^2[6b^2\rho^4+(3b^2+c^2)\rho^2-2c^2]}{8(c^2+b^2\rho^2)^3(1+\rho^2)}
\eeq
with $$\tau=\int \alpha d\rho=\int \sqrt{\frac{2(c^2+b^2\rho^2)}{1+\rho^2}}d\rho.$$
\end{prop}

\begin{proof} Indeed,
$$\frac{d\psi_1}{d\rho}=\frac{d\psi_1}{d\tau}\frac{d\tau}{d\rho}=\alpha\frac{d\psi_1}{d\tau},\ \
\frac{d^2\psi_1}{d\rho^2}=\alpha^2\frac{d^2\psi_1}{d\tau^2}+\alpha\frac{d\alpha}{d\tau}\frac{d\psi_1}{d\tau},
$$
then we have
\beq{in}
-\frac{d^2\psi_1}{d\tau^2}-\frac{2a^2\rho}{\alpha^3(1+\rho^2)^2}\frac{d\psi_1}{d\tau}+\frac{k^2}{2a^2(1+\rho^2)}\psi_1=E\psi_1.
\eeq
Taking $\psi_1(\rho)=\alpha^{-\frac{1}{2}}\varphi$ in \eqref{in}, a direct calculation yields \eqref{sh},\eqref{pot}.
Note that at infinity the new variable $\tau$ grows as $b\rho$.
\end{proof}

The classical reflection condition $I^2>2a^2H$ can be rewritten in quantum case as
\beq{qref}
E<\frac{k^2}{2a^2}
\eeq
since $k$ is the quantised value of $I$ and $E$ is quantum energy $H.$

It is interesting that the right hand side is the height of the graph of the first part of the potential (\ref{pot}).
The second part, decaying as $\rho^{-4}$ as $\rho\to \infty$, for some reason does not appear in this condition.

Note that if let $\rho=\sinh u $ in \eqref{Lk}, we get
\beq{ince}
\frac{1}{2(c^2+b^2\sinh^2 u)}\left(-\frac{d^2\psi_1}{du^2}+\frac{\sinh u}{\cosh u}\frac{d\psi_1}{du}\right)
+\frac{k^2}{2a^2\cosh^2 u}\psi_1=E\psi_1,
\eeq
which leads to
\beq{mince}
(1+\cosh2u)\frac{d^2\psi_1}{du^2}-\sinh2u\frac{d\psi_1}{du}
+(\alpha + \beta \cosh 2u + \gamma \cosh 4u)\psi_1=0,
\eeq
where
\beq{abg}
\alpha=\frac{a^2(4a^2-3b^2)E+8k^2(b^2-2a^2)}{2a^2},\, \beta=\frac{8k^2b^2-4a^2c^2E}{2a^2},\,\gamma=\frac{b^2E}{2}
\eeq
This equation is a mixture of the hyperbolic versions of the generalised Ince equation
$$
(1+\alpha \cos 2x) y''+\beta\sin 2x y'+(\gamma+\delta \cos 2x)y=0
$$
and the Whittaker-Hill equation
$$
y''+(\alpha+ \beta \cos 2x + \gamma \cos 4x)y=0
$$
(see \cite{MW}). It would be interesting to study this link in more detail.

\section{Kn\"orrer's map and projective equivalence}

Let us consider now the corresponding projective quadric $\mathcal H$ given in the real projective space $\mathbb RP^{n+1}$ with projective coordinates $\xi_0:\dots :\xi_n: \xi_{n+1}$ by the equation
\beq{ph}
b_0\xi_0^2 + \dots + b_{n}\xi_{n}^2 = \xi_{n+1}^2,
\eeq
where we assume as before that $b_0<0<b_1<\dots<b_n$.
In the affine chart $\mathbb R^n$ with $\xi_{n+1}\neq 0$ with coordinates $x_k=\xi_k/\xi_{n+1}, \, k=0,\dots, n$ it coincides with our hyperboloid
$$
b_0x_0^2 + \dots + b_{n}x_{n}^2=1,
$$
so $\mathcal H$ is the projective closure of our hyperboloid.

Topologically $\mathcal H$ is the quotient of $S^{n-1}\times S^1$ by the involution $\sigma$ acting by the antipodal map $v \to-v$ on both $S^{n-1}$ and $S^1$. When $n$ is even $\mathcal H$ is diffeomorphic $S^{n-1}\times S^1$, but if $n$ is odd and larger than 1 $\mathcal H$ is non-orientable manifold, doubly covered by $S^{n-1}\times S^1$.

Note that the projective closure $\mathcal H_2$ of the two-sheeted hyperboloid
\beq{ph1}
b_0\xi_0^2 + \dots + b_{n}\xi_{n}^2 = -\xi_{n+1}^2,
\eeq
topologically is just the sphere $S^n.$

The standard Euclidean metric $ds_0^2=\sum_{k=0}^{n} dx_k^2= (dx, dx)$ on the $\mathbb R^{n+1}$
in the projective coordinates has the form
\beq{metric0}
 ds_0^2=\sum_{k=0}^n \frac{\xi_{n+1}d\xi_k^2-2\xi_kd\xi_kd\xi_{n+1}}{\xi_{n+1}^3}+ \frac{|\xi|^2 d\xi_{n+1}^2}{\xi_{n+1}^{4}},\,\, |\xi|^2=\sum_{k=0}^n \xi_k^2.
\eeq
It becomes singular at the infinity hyperplane $\xi_{n+1}=0$ and can not be used to induce regular metric on the whole $\mathcal H.$

Kn\"orrer's change of time (\ref{tau}) and the form (\ref{lambda}) of the Lagrange multiplier $\lambda=\frac{(Bx',x')}{(Bx,Bx)}$ suggest the following conformally flat metric given in the affine coordinates by
\beq{metric1}
ds_1^2=\frac{1}{|Bx|^2}(Bdx,dx)=\frac{b_0 dx_0^2 + b_1 dx_1^2 + \dots + b_{n} dx_{n}^2}{b_0^2x_0^2 + b_1^2 x_1^2 \dots + b_{n}^2 x_{n}^2}.
\eeq

In the case of ellipsoid $\mathcal E$ given by (\ref{ph}) with positive $b_0,\dots, b_n$ Tabachnikov \cite{T,T2} and, independently, Matveev and Topalov \cite{MT,TM} proved the following remarkable property of this metric.

Recall that two metrics on a manifold are {\it projectively equivalent} if they have the same set of geodesics, possibly in different parametrisation.
Such metrics were studied already by Beltrami \cite{Belt}, Dini \cite{Dini} and Levi-Civita \cite{LC}, but in spite of long history the following nice result seems to be discovered only recently.

\begin{Theorem}\cite{MT,T}
The restrictions of metrics $ds_0^2$ and $ds_1^2$ on $\mathcal E$ are projectively equivalent.
\end{Theorem}

We have the following additional observation, which seems to be new.

\begin{Theorem}
The ellipsoid $\mathcal E$ is totally geodesic submanifold of $\mathbb R^{n+1}$ with metric $ds_1^2$ given by (\ref{metric1}).
The Kn\"orrer parameter $\tau$ given by (\ref{tau}) is a natural parameter on the corresponding geodesics.
\end{Theorem}

\begin{proof}
The easiest geometric proof is based on the following lemma.

\begin{Lemma}
The involution $\sigma: \mathbb R^{n+1} \to \mathbb R^{n+1}$ defined by
\beq{sigma}
y=\sigma(x):=\frac{x}{(Bx,x)}
\eeq
preserves the metric (\ref{metric1}) and leaves the ellipsoid $\mathcal E$ as fixed point set.
\end{Lemma}
Proof is by direct check. The facts that $\sigma$ is an involution and has the ellipsoid $\mathcal E$ fixed is obvious.
To prove the rest we have
$$
dy=\frac{dx}{(Bx,x)}-\frac{2x(Bx,dx)}{(Bx,x)^2},
$$
$$
(Bdy,dy)=\frac{(Bdx,dx)}{(Bx,x)^2}-\frac{4(Bx,dx)^2}{(Bx,x)^3}+\frac{4(Bx,x)(Bx,dx)^2}{(Bx,x)^4}=\frac{(Bdx,dx)}{(Bx,x)^2}.
$$
Since $|By|^2=\frac{|Bx|^2}{(Bx,x)^2}$ we have
$$
\frac{(Bdy,dy)}{|By|^2}=\frac{(Bdx,dx)}{(Bx,x)^2}\frac{(Bx,x)^2}{|Bx|^2}=\frac{(Bdx,dx)}{|Bx|^2},
$$
which means that $\sigma$ is an isometry of (\ref{metric1}). Now the first claim follows from the well-known fact that the fixed set of an isometry is totally geodesic submanifold.

To prove the second part recall that the geodesics on a submanifold $Y^n \subset X^{n+1}$ of a pseudo-Riemannian manifold $X^{n+1}$ are the curves $x(s)$ on $Y^n$ with the second covariant derivative $\nabla^2 x$ being normal to $Y^n$ for all $s.$
By definition, the second covariant derivative is defined in the local coordinates $x^i, \, i=0,\dots, n$ on $X^{n+1}$ by
$$
\nabla^2 x^i=\ddot x^i+\Gamma^{i}_{jk} \dot{x}^j \dot{x}^k,
$$
where $\Gamma^{i}_{jk}$ are the standard Christoffel symbols  \cite{DC} of the metric $g_{ij}$ on $X^{n+1}$ defined by
$$
\Gamma^{i}_{jk}=\frac{1}{2}g^{im}(\partial_j g_{mk}+\partial_k g_{mj}-\partial_m g_{jk}).
$$

An easy calculation for the metric (\ref{metric1}) shows that $\Gamma^{i}_{jk}=0$ for distinct $i,j,k$ and
$$
\Gamma^{i}_{ik}= \partial_k \log \gamma, \,\, k=0,\dots,n, \quad
\Gamma^{i}_{jj}=-\frac{b_j}{b_i} \partial_i \log \gamma, \,\,  j \neq i,
$$
where as before $ \gamma=|Bx|^{-1}.$ Using this, we can rewrite
$$
\Gamma^{i}_{jk} \dot{x}^j \dot{x}^k=2\frac{\dot \gamma}{\gamma}\dot x^i+\frac{(B\dot x, \dot x)}{|Bx|^2}(Bx)^i,
$$
and thus
$$
\nabla^2 x=\ddot x+2\frac{\dot \gamma}{\gamma}\dot x+\frac{(B\dot x, \dot x)}{|Bx|^2}Bx.
$$
Comparing this with equation (\ref{kgeod}) of geodesics in Kn\"orrer's parameter we see that they satisfy the geodesic equation of metric (\ref{metric1})
$$
\nabla^2 x=\ddot x+2\frac{\dot \gamma}{\gamma}\dot x+Bx=0,
$$
since in Kn\"orrer's parameter $\frac{(B\dot x, \dot x)}{|Bx|^2}= \varepsilon$ and $\varepsilon=1$ in the ellipsoid case.
This completes the proof of the theorem.
\end{proof}

\begin{Remark}
Note that this gives one more proof of Tabachnikov-Matveev-Topalov result, which works also for hyperboloids. The only difference is that in one-sheeted case the second metric is pseudo-Riemannian and the isotropic geodesics (which are simply generating straight lines) should be considered separately.
\end{Remark}

Let us come back to the projective picture.
At the infinity hyperplane $\xi_{n+1}=0$ we have a clear problem with the first metric.
Remarkably the second metric has also nice properties at infinity as it was pointed out by the first author in \cite{V1}.

\begin{Theorem}\cite{V1}
The restriction of metric (\ref{metric1}) to the projective closure of hyperboloids is a regular metric, which is Riemannian in two-sheeted case $\mathcal H_2$ and pseudo-Riemannian of signature $(n,1)$ in one-sheeted case
$\mathcal H.$
\end{Theorem}

\begin{proof}
We can check this by explicit calculations, which we will perform in one-sheeted case. Consider for example the affine chart with $\xi_0 \neq 0$
with affine coordinates $y_k = \xi_k/\xi_{0}, \, k =0,\,\dots,\, n.$ Making the change of variables
$$x_0=\frac{1}{y_{n+1}}, \,x_1=\frac{y_1}{y_{n+1}}, \dots, \, x_{n}=\frac{y_{n}}{y_{n+1}},$$
we have the following affine equation
$$b_0 + b_1 y_1^2 \dots + b_{n}y_{n}^2 = y_{n+1}^2.$$

Now substituting into (\ref{metric1})
$$
dx_0=-\frac{dy_{n+1}}{y_{n+1}^2}, \, dx_1=\frac{y_{n+1} dy_1-y_1 dy_{n+1}}{y_{n+1}^2}, \dots, dx_n=\frac{y_{n+1} dy_n-y_n dy_{n+1}}{y_{n+1}^2}
$$
and using the relation
$$
y_{n+1}dy_{n+1}=b_1y_1dy_1+\dots + b_ny_n dy_n
$$
holding on $\mathcal H$, we have the following expression for the restricted metric $dr^2$ on $\mathcal H$:
\begin{equation}
\label{drtil}
dr^2 = \frac{b_1 dy_1^2 + \dots + b_{n} dy_{n}^2-dy_{n+1}^2}{b_0^2 + b_1^2 y_1^2+ \dots + b_{n}^2 y_{n}^2}.
\end{equation}
We see indeed that this metric is regular when $y_{n+1}=0,$ which means that $ds_1^2$ indeed can be regularly extended to the whole $\mathcal H.$
(Note that outside $\mathcal H$ the metrics $dr^2$ and $ds_1^2$ are different.) The same works in any other affine chart, proving the claim in one-sheeted case.

In the two-sheeted case $\mathcal H_2$ we have similar calculations leading to the restriction of the Riemannian metric
\begin{equation}
\label{drtil2}
dr^2 = \frac{b_1 dy_1^2 + \dots + b_{n} dy_{n}^2+dy_{n+1}^2}{b_0^2 + b_1^2 y_1^2+ \dots + b_{n}^2 y_{n}^2}.
\end{equation}
Note that this provides a Riemannian metric on the topological sphere $S^n$ for which projectively equivalent metric (\ref{metric0}) is singular.
\end{proof}

Let us define now the following projective version of Kn\"orrer's map $$\nu: \mathcal H \to \mathbb RP^{n}.$$
In the affine chart with $\xi_{n+1}\neq 0$ we define it by the formula
\begin{equation}
\label{pk}
\nu(x)=[Bx], \quad x \in \mathbb H,
\end{equation}
where $[Bx] \in \mathbb RP^{n}$ is the line defined by vector $Bx$, and then extend it to the whole $\mathcal H$ by continuity. It is easy to check that this provides a smooth map of $\mathcal H$ to $\mathbb RP^{n}.$

Since Neumann system on $S^n$ is invariant under antipodal map $\sigma: v \to -v$ it can be reduced to the quotient $\mathbb RP^{n}=S^n/\sigma.$

\begin{Theorem}
The projective Kn\"orrer's map (\ref{pk}) maps the non-isotropic geodesics on $\mathcal H$ with restricted metric (\ref{metric1}) to the solutions of the Neumann system on $\mathbb RP^n,$ satisfying the relation $\Psi^\varepsilon_0(p,q)=0.$
\end{Theorem}

The proof follows directly from the proofs of Theorems 4.1 and 9.2. Similar claim holds for two-sheeted case.



%

\section{Concluding remarks}

The quantum scattering on hyperboloids is probably the most important question, which is still to be studied even in two-dimensional case.
The calculations in the symmetric case above show that we probably should not expect an easy explicit answer.

Another interesting question is related to Moser-Trubowitz isomorphism between solutions of Neumann systems and finite-gap potentials \cite{M1, M3, V80}.
As it was observed in \cite{V} Kn\"orrer's map allows to translate it into relation between geodesics on ellipsoid with certain stability problem in particle dynamics.
A similar observation about relation of geodesics on ellipsoid with the stationary Harry-Dym equation was done by Cao \cite{Cao}.

It is natural to ask about possible spectral interpretation of geodesics on hyperboloids (see \cite{V1}).
Our results show that after Kn\"orrer's map and Moser-Trubowitz isomorphism we have a finite-gap Schr\"odinger operator restricted to certain finite interval, which seems to be not studied yet.

Finally we can mention the problem of geodesic scattering on quadrics in pseudo-Euclidean case (see \cite{V90} and relevant work by Khesin and Tabachnikov \cite{KT} and Dragovic and Radnovic \cite{DR}).

\section{Acknowledgements}

We are very grateful to Alexey Bolsinov, Jonathan Eckhardt, Jenya Ferapontov and Sergei Tabachnikov for helpful and encouraging discussions.

L.W. is grateful to the Department of Mathematical Sciences, Loughborough University for the hospitality during the academic year 2019-20, when the work had been done.

L.W. was supported by National Natural Science Foundation of China (project no.11871232) and China Scholarship Council.

\end{document}